\documentclass{amsart}
\usepackage{amssymb}

\usepackage{mathrsfs}
\newtheorem{thm}{Theorem}

\newtheorem{lem}{Lemma}

\theoremstyle{definition}

\theoremstyle{remark}

\begin{document}

\title[A Curious congruence involving alternating harmonic sums]
{A Curious congruence involving alternating harmonic sums}

\author{ LIUQUAN WANG}

\address{Department of Mathematics, National University of Singapore, Singapore, 119076, Singapore}

\email{mathlqwang@163.com}

\date{June 2, 2014}



\dedicatory{}


\begin{abstract}
Let $p$ be a prime and ${\mathcal{P}_{p}}$  the set of positive integers which are prime to $p$. We establish the following interesting congruence
\[\sum\limits_{\begin{smallmatrix}
 i+j+k={{p}^{r}} \\
 i,j,k\in {\mathcal{P}_{p}}
\end{smallmatrix}}{\frac{{{(-1)}^{i}}}{ijk}}\equiv \frac{{{p}^{r-1}}}{2}{{B}_{p-3}}\, (\bmod \, {{p}^{r}}).\]
\end{abstract}

\maketitle

\section{Introduction and Main Results}

Let ${{B}_{n}}$ be the $n$-th Bernoulli number, which is defined by
\[\frac{x}{{{e}^{x}}-1}=\sum\limits_{n=0}^{\infty }{\frac{{{B}_{n}}}{n}{{x}^{n}}}.\]
In  \cite{Zhao}, Zhao found a curious congruence
\begin{equation}\label{zhao1}
\sum\limits_{\begin{smallmatrix}
 i+j+k=p \\
 i,j,k>0
\end{smallmatrix}}^{{}}{\frac{1}{ijk}\equiv -2{{B}_{p-3}}\,(\bmod \,p),}
\end{equation}
where $p\ge 3$ is a prime.

Let ${\mathcal{P}_{n}}$ denotes the set of positive integers which are prime to  $n$.  Recently, Wang and Cai \cite{Wang} gave a generalization to this congruence, they have shown
\[\sum\limits_{\begin{smallmatrix}
 i+j+k={{p}^{r }} \\
 i,j,k\in {\mathcal{P}_{p}}
\end{smallmatrix}}^{{}}{\frac{1}{ijk}\equiv -2{{p}^{r-1}}{{B}_{p-3}}\,(\bmod \,p^r)}.\]
For more variant and generalizations of (\ref{zhao1}), we refer the reader to see \cite{Xia} and \cite{Zhou}.

In this paper, we consider the alternating sums and obtain the following analogous result.
\begin{thm}\label{thm1}
 Let $p\ge 3$ be a prime and $r$  a positive integer, then
\begin{equation}
\sum\limits_{\begin{smallmatrix}
 i+j+k={{p}^{r}} \\
 i,j,k\in {\mathcal{P}_{p}}
\end{smallmatrix}}{\frac{{{(-1)}^{i}}}{ijk}}\equiv \frac{{{p}^{r-1}}}{2}{{B}_{p-3}}\, (\bmod \, {{p}^{r}}).
\end{equation}
\end{thm}
Furthermore, we have
\begin{thm}\label{thm2}
Let $n$ be a positive integer and $p\ge 3$  a prime, if ${{p}^{r}}|n$,  $r\ge 1$ is an integer, then
\[\sum\limits_{\begin{smallmatrix}
 i+j+k=n \\
 i,j,k\in {\mathcal{P}_{p}}
\end{smallmatrix}}^{{}}{\frac{{(-1)}^{i}}{ijk}\equiv \frac{n}{2p}{{B}_{p-3}}\,(\bmod \, p^r).}\]
\end{thm}
In particular, if $n=p^{r}$, Theorem \ref{thm2} becomes Theorem \ref{thm1}.

\section{Preliminaries}

\begin{lem}[cf. Theorem 5.2 in \cite{Sun}]\label{px3}
For any prime $p \ge 3$ we have
\[\sum\limits_{x=1}^{(p-1)/2}{\frac{1}{{{x}^{3}}}}\equiv -2{{B}_{p-3}} \, (\bmod \, p).\]
\end{lem}

The following Lemma is contained in \cite{Wang}, for the sake of completeness, we provide the proof here.
\begin{lem}\label{Sxp}
Let $p\ge 3$ be a prime and $r$ a positive integer, for $1\le x\le p-1$, define
\[S(x,{{p}^{r}})=\sum\limits_{\begin{smallmatrix}
 i\equiv x\,(\bmod p) \\
 1\le i\le {{p}^{r}}-1
\end{smallmatrix}}^{{}}{\frac{1}{i}},\]
we have
\[{S(x,{{p}^{r}})\equiv p^{r-1}x^{-1} \,(\bmod \,{{p}^{r}})}.\]

\end{lem}

\begin{proof}
For any integer  $0\le k\le {{p}^{r}}-1$, write $k$ as
\[k=a+{{p}^{r-1}}b,0\le a\le {{p}^{r-1}}-1,0\le b\le p-1.\]
Then
\begin{equation}\label{equal}
\begin{split}
  & S(x,{{p}^{r+1}})-pS(x,{{p}^{r}}) \\
 & =\sum\limits_{k=0}^{{{p}^{r}}-1}{\frac{1}{x+kp}-p\sum\limits_{k=0}^{{{p}^{r-1}}-1}{\frac{1}{x+kp}}} \\
 & =\sum\limits_{b=0}^{p-1}{\sum\limits_{a=0}^{{{p}^{r-1}}-1}{\frac{1}{x+(a+{{p}^{r-1}}b)p}}}-\sum\limits_{b=0}^{p-1}{\sum\limits_{a=0}^{{{p}^{r-1}}-1}{\frac{1}{x+ap}}} \\
 & =-{{p}^{r}}\sum\limits_{b=1}^{p-1}{\sum\limits_{a=0}^{{{p}^{r-1}}-1}{\frac{b}{(x+ap)\{x+(a+{{p}^{r-1}}b)p\}}}}.
\end{split}
\end{equation}

Because
\[\sum\limits_{b=0}^{p-1}{b}=\frac{p(p-1)}{2}\equiv 0\,\,(\bmod \,p),\]
we have $S(x,{{p}^{r+1}})-pS(x,p^{r})\equiv 0 \,(\bmod \, {{p}^{r+1}})$.

Note that $S(x,p)={{x}^{-1}}$, by induction on $r$  it's easy to see that $S(x,{{p}^{r}})\equiv {{p}^{r-1}}{{x}^{-1}} \, (\bmod \, {{p}^{r}}).$
\end{proof}


\section{Proofs of the Theorems}
\begin{proof}[Proof of Theorem 1]
It's easy to see that
\begin{equation}\label{start}
\begin{split}
 \sum\limits_{\begin{smallmatrix}
 i+j+k={{p}^{r}} \\
 i,j,k\in {\mathcal{P}_{p}}
\end{smallmatrix}}{\frac{{{(-1)}^{i}}}{ijk}} &=\sum\limits_{\begin{smallmatrix}
 i+j+k={{p}^{r}} \\
 i,j,k\in {\mathcal{P}_{p}}
\end{smallmatrix}}{\frac{{{(-1)}^{{{p}^{r}}-j-k}}}{({{p}^{r}}-j-k)jk}} \\
 & \equiv \sum\limits_{\begin{smallmatrix}
 j+k<{{p}^{r}} \\
 j,k,j+k\in {\mathcal{P}_{p}}
\end{smallmatrix}}{\frac{{{(-1)}^{j+k}}}{jk(j+k)}} \\
 & \equiv 2\sum\limits_{\begin{smallmatrix}
 1\le j<m<{{p}^{r}} \\
 j,m,m-j\in {\mathcal{P}_{p}}
\end{smallmatrix}}{\frac{{{(-1)}^{m}}}{j{{m}^{2}}}} \, (\bmod \, p^r),
\end{split}
\end{equation}
where in the last equality we replace $j+k$ by $m$ and used the fact
\[\frac{1}{jk(j+k)}=\frac{j+k}{jk{{(j+k)}^{2}}}=\big(\frac{1}{j}+\frac{1}{k}\big)\frac{1}{{{(j+k)}^{2}}}.\]
 Note that
\[\sum\limits_{\begin{smallmatrix}
 1\le j<m<{{p}^{r}} \\
 j,m,m-j\in {\mathcal{P}_{p}}
\end{smallmatrix}}{\frac{{{(-1)}^{m}}}{j{{m}^{2}}}}=\sum\limits_{\begin{smallmatrix}
 1\le j<m<{{p}^{r}} \\
 j,m,m-j\in {\mathcal{P}_{p}}
\end{smallmatrix}}{\frac{{{(-1)}^{{{p}^{r}}-j}}}{({{p}^{r}}-m){{(p^{r}-j)}^{2}}}}\equiv \sum\limits_{\begin{smallmatrix}
 1\le j<m<{{p}^{r}} \\
 j,m,m-j\in {\mathcal{P}_{p}}
\end{smallmatrix}}{\frac{{{(-1)}^{j}}}{m{{j}^{2}}}} \, (\bmod \, p^r).\]
Hence we have
\begin{equation}\label{mid}
\begin{split}
  2\sum\limits_{\begin{smallmatrix}
 1\le j<m<{{p}^{r}} \\
 j,m,m-j\in {\mathcal{P}_{p}}
\end{smallmatrix}}{\frac{{{(-1)}^{m}}}{j{{m}^{2}}}}& =\sum\limits_{\begin{smallmatrix}
 1\le j,m<{{p}^{r}} \\
 j,m,m-j\in {\mathcal{P}_{p}}
\end{smallmatrix}}{\frac{{{(-1)}^{m}}}{j{{m}^{2}}}} \\
 & =\sum\limits_{\begin{smallmatrix}
 1\le j,m<{{p}^{r}} \\
 j,m\in {\mathcal{P}_{p}}
\end{smallmatrix}}{\frac{{{(-1)}^{m}}}{j{{m}^{2}}}}-\sum\limits_{\begin{smallmatrix}
 1\le j,m<{{p}^{r}} \\
 j\in {\mathcal{P}_{p}},j\equiv m \, (\bmod \, p)
\end{smallmatrix}}{\frac{{{(-1)}^{m}}}{j{{m}^{2}}}} \\
 & =\Big(\sum\limits_{\begin{smallmatrix}
 1\le j<{{p}^{r}} \\
 j\in {\mathcal{P}_{p}}
\end{smallmatrix}}{\frac{1}{j}}\Big)\Big(\sum\limits_{\begin{smallmatrix}
 1\le m<{{p}^{r}} \\
 m\in {\mathcal{P}_{p}}
\end{smallmatrix}}{\frac{{{(-1)}^{m}}}{{{m}^{2}}}}\Big)-\sum\limits_{\begin{smallmatrix}
 1\le j,m<{{p}^{r}} \\
 j\in {\mathcal{P}_{p}},j\equiv m \, (\bmod \, p)
\end{smallmatrix}}{\frac{{{(-1)}^{m}}}{j{{m}^{2}}}} \\
 & \equiv -\sum\limits_{\begin{smallmatrix}
 1\le j,m<{{p}^{r}} \\
 j\in {\mathcal{P}_{p}},j\equiv m \, (\bmod \, p)
\end{smallmatrix}}{\frac{{{(-1)}^{m}}}{j{{m}^{2}}}} \, (\bmod \, {{p}^{r}}),
\end{split}
\end{equation}
here  the last congruence equality follows from the fact
\[\sum\limits_{\begin{smallmatrix}
 1\le j<{{p}^{r}} \\
 j\in {\mathcal{P}_{p}}
\end{smallmatrix}}{\frac{1}{j}}=\frac{1}{2}\sum\limits_{\begin{smallmatrix}
 1\le j<{{p}^{r}} \\
 j\in {\mathcal{P}_{p}}
\end{smallmatrix}}{(\frac{1}{j}+\frac{1}{{{p}^{r}}-j})}=\frac{1}{2}\sum\limits_{\begin{smallmatrix}
 1\le j<{{p}^{r}} \\
 j\in {\mathcal{P}_{p}}
\end{smallmatrix}}{\frac{{{p}^{r}}}{j({{p}^{r}}-j)}}\equiv 0 \, (\bmod \, {{p}^{r}}).\]

 By Lemma \ref{px3} we deduce that
\[\sum\limits_{x=1}^{p-1}{\frac{{{(-1)}^{x}}}{{{x}^{3}}}}=-\sum\limits_{x=1}^{p-1}{\frac{1}{{{x}^{3}}}}+2\sum\limits_{x=1}^{\frac{p-1}{2}}{\frac{1}{{{(2x)}^{3}}}}\equiv -\frac{1}{2}{{B}_{p-3}} \, (\bmod  \, p).\]
Thus by Lemma \ref{Sxp} we have
\begin{equation}\label{end}
\begin{split}
   \sum\limits_{\begin{smallmatrix}
 1\le j,m<{{p}^{r}} \\
 j\in {\mathcal{P}_{p}},j\equiv m \, (\bmod \, p)
\end{smallmatrix}}{\frac{{{(-1)}^{m}}}{j{{m}^{2}}}}&=\sum\limits_{x=1}^{p-1}{\Big(\sum\limits_{\begin{smallmatrix}
 1\le j<{{p}^{r}} \\
 j\equiv x \, (\bmod  \, p)
\end{smallmatrix}}{\frac{1}{j}}\Big)\Big(\sum\limits_{\begin{smallmatrix}
 1\le m<{{p}^{r}} \\
 m\equiv x \, (\bmod \,  p)
\end{smallmatrix}}{\frac{{{(-1)}^{m}}}{{{m}^{2}}}}\Big)} \\
 & ={{{p}^{r-1}}}\sum\limits_{x=1}^{p-1}{\frac{1}{x}\sum\limits_{\begin{smallmatrix}
 1\le m<{{p}^{r}} \\
 m\equiv x \, (\bmod  \, p)
\end{smallmatrix}}{\frac{{{(-1)}^{m}}}{{{m}^{2}}}}} \\
 & = {{{p}^{r-1}}} \sum\limits_{x=1}^{p-1}{\frac{1}{x}\sum\limits_{a=0}^{{{p}^{r-1}}-1}{\frac{{{(-1)}^{x+ap}}}{{{(x+ap)}^{2}}}}} \\
 & \equiv  {{{p}^{r-1}}} \sum\limits_{x=1}^{p-1}{\frac{{{(-1)}^{x}}}{{{x}^{3}}}} \\
 & \equiv -\frac{{{p}^{r-1}}}{2}{{B}_{p-3}} \, (\bmod  \, {{p}^{r}}). \\
\end{split}
\end{equation}

Combining (\ref{start}), (\ref{mid}) and (\ref{end}), we complete the proof of Theorem \ref{thm1}.
\end{proof}

\begin{proof}[Proof of Theorem 2]
Let $n={{p}^{r}}m, $  we only need to show
\begin{equation}\label{dengjia}
\sum\limits_{\begin{smallmatrix}
 i+j+k=m{{p}^{r}} \\
 i,j,k\in {\mathcal{P}_{p}}
\end{smallmatrix}}^{{}}{\frac{(-1)^{i}}{ijk}\equiv \frac{m{{p}^{r-1}}}{2}{{B}_{p-3}}\,(\bmod \, p^{r}).}
\end{equation}
For every triple $(i,j,k)$ of positive integers  which satisfies $i+j+k=m{{p}^{r}},i,j,k\in {\mathcal{P}_{p}}$, we rewrite
\[i=x{{p}^{r}}+{{i}_{0}},j=y{{p}^{r}}+{{j}_{0}},k=z{{p}^{r}}+{{k}_{0}},\]
where $1\le {{i}_{0}},{{j}_{0}},{{k}_{0}}\le {{p}^{r}}-1$ are prime to $p$ and $x, y, z$ are nonnegative integers.

Since  $3\le {{i}_{0}}+{{j}_{0}}+{{k}_{0}}<3{{p}^{r}}$ and ${{i}_{0}}+{{j}_{0}}+{{k}_{0}}=(m-x-y-z){{p}^{r}},$ it's easy to see that
\begin{displaymath}
\left\{ \begin{array}{l}
 {{i}_{0}}+{{j}_{0}}+{{k}_{0}}={{p}^{r}} \\ \nonumber
x+y+z=m-1
\end{array} \right.
\quad \textrm{or} \quad \left\{ \begin{array}{l}
{{i}_{0}}+{{j}_{0}}+{{k}_{0}}=2{{p}^{r}} \\ \nonumber
 x+y+z=m-2
\end{array} \right..
\end{displaymath}
Note that
\begin{displaymath}
\begin{split}
\sum\limits_{\begin{smallmatrix}
 x+y+z=m-1 \\
 x,y,z\ge 0
\end{smallmatrix}}{{{(-1)}^{x}}}&=\sum\limits_{k=0}^{m-1}{{{(-1)}^{k}}(m-k)}=\big[\frac{m+1}{2}\big],\\
\sum\limits_{\begin{smallmatrix}
 x+y+z=m-2 \\
 x,y,z\ge 0
\end{smallmatrix}}{{{(-1)}^{x}}}&=\sum\limits_{k=0}^{m-2}{{{(-1)}^{k}}(m-1-k)}=\big[\frac{m}{2}\big],
\end{split}
\end{displaymath}
here $[x]$ denotes the integer part of $x$. Hence
\begin{equation}\label{thm2mid}
\begin{split}
  \sum\limits_{\begin{smallmatrix}
 i+j+k=m{{p}^{r}} \\
 i,j,k\in {\mathcal{P}_{p}}
\end{smallmatrix}}{\frac{(-1)^{i}}{ijk}}& =\sum\limits_{\begin{smallmatrix}
 {{i}_{0}}+{{j}_{0}}+{{k}_{0}}={{p}^{r}} \\
 x+y+z=m-1 \\
  {{i}_{0}},{{j}_{0}},{{k}_{0}}\in {\mathcal{P}_{p}}
\end{smallmatrix}}{\frac{(-1)^{x{{p}^{r}}+{{i}_{0}}}}{(x{{p}^{r}}+{{i}_{0}})(y{{p}^{r}}+{{j}_{0}})(z{{p}^{r}}+{{k}_{0}})}} \\
& \quad \quad +\sum\limits_{\begin{smallmatrix}
 {{i}_{0}}'+{{j}_{0}}'+{{k}_{0}}'=2{{p}^{r}} \\
 x+y+z=m-2\\
  {{i}_{0}}',{{j}_{0}}',{{k}_{0}}'\in {\mathcal{P}_{p}}\\
   {{i}_{0}}',{{j}_{0}}',{{k}_{0}}' <p^{r}
\end{smallmatrix}}{\frac{(-1)^{x{{p}^{r}}+{{i}_{0}}'}}{(x{{p}^{r}}+{{i}_{0}}')(y{{p}^{r}}+{{j}_{0}}')(z{{p}^{r}}+{{k}_{0}}')}} \\
 & \equiv \sum\limits_{\begin{smallmatrix}
 {{i}_{0}}+{{j}_{0}}+{{k}_{0}}={{p}^{r}} \\
 x+y+z=m-1\\
  {{i}_{0}},{{j}_{0}},{{k}_{0}}\in {\mathcal{P}_{p}}
\end{smallmatrix}}{\frac{(-1)^{x+{{i}_{0}}}}{{{i}_{0}}{{j}_{0}}{{k}_{0}}}}+\sum\limits_{\begin{smallmatrix}
 {{i}_{0}}'+{{j}_{0}}'+{{k}_{0}}'=2{{p}^{r}} \\
 x+y+z=m-2 \\
  {{i}_{0}}',{{j}_{0}}',{{k}_{0}}'\in {\mathcal{P}_{p}}\\
   {{i}_{0}}',{{j}_{0}}',{{k}_{0}}' <p^{r}
\end{smallmatrix}}{\frac{(-1)^{x+{{{i}_{0}}'}}}{{{i}_{0}}'{{j}_{0}}'{{k}_{0}}'}}  \\
& \equiv   \big[\frac{m+1}{2}\big]\sum\limits_{\begin{smallmatrix}
 {{i}_{0}}+{{j}_{0}}+{{k}_{0}}={{p}^{r}} \\
 {{i}_{0}},{{j}_{0}},{{k}_{0}}\in {\mathcal{P}_{p}}
\end{smallmatrix}}{\frac{1}{{{i}_{0}}{{j}_{0}}{{k}_{0}}}}+ \big[\frac{m}{2}\big]\sum\limits_{\begin{smallmatrix}
 {{i}_{0}}'+{{j}_{0}}'+{{k}_{0}}'=2{{p}^{r}} \\
 {{i}_{0}}',{{j}_{0}}',{{k}_{0}}'\in {\mathcal{P}_{p}} \\
 {{i}_{0}}',{{j}_{0}}',{{k}_{0}}' <p^{r}
\end{smallmatrix}}{\frac{(-1)^{{{i}_{0}}'}}{{{i}_{0}}'{{j}_{0}}'{{k}_{0}}'}} \,(\bmod \, p^r).
\end{split}
\end{equation}

For the second sum in (\ref{thm2mid}), since $({{i}_{0}},{{j}_{0}},{{k}_{0}})\leftrightarrow ({{p}^{r}}-{{i}_{0}},{{p}^{r}}-{{j}_{0}},{{p}^{r}}-{{k}_{0}})$ gives a bijection between the solutions of  ${{i}_{0}}+{{j}_{0}}+{{k}_{0}}={{p}^{r}}$  and ${{i}_{0}}'+{{j}_{0}}'+{{k}_{0}}'=2{{p}^{r}}$ (here $i_{0}',j_{0}',k_{0}'<p^r$), we have
\begin{equation}\label{2ndsum}
\begin{split}
\sum\limits_{\begin{smallmatrix}
 {{i}_{0}}'+{{j}_{0}}'+{{k}_{0}}'=2{{p}^{r}} \\
 {{i}_{0}}',{{j}_{0}}',{{k}_{0}}'\in {\mathcal{P}_{p}}\\
  {{i}_{0}}',{{j}_{0}}',{{k}_{0}}' <p^{r}
\end{smallmatrix}}{\frac{(-1)^{{{i}_{0}}'}}{{{i}_{0}}'{{j}_{0}}'{{k}_{0}}'}} & =\sum\limits_{\begin{smallmatrix}
 {{i}_{0}}+{{j}_{0}}+{{k}_{0}}={{p}^{r}} \\
 {{i}_{0}},{{j}_{0}},{{k}_{0}}\in {\mathcal{P}_{p}}
\end{smallmatrix}}{\frac{(-1)^{p^{r}-{{i}_{0}}}}{({{p}^{r}}-{{i}_{0}})({{p}^{r}}-{{j}_{0}})({{p}^{r}}-{{k}_{0}})}}\\
&\equiv  \sum\limits_{\begin{smallmatrix}
 {{i}_{0}}+{{j}_{0}}+{{k}_{0}}={{p}^{r}} \\
 {{i}_{0}},{{j}_{0}},{{k}_{0}}\in {\mathcal{P}_{p}}
\end{smallmatrix}}{\frac{(-1)^{i_{0}}}{{{i}_{0}}{{j}_{0}}{{k}_{0}}}}\,(\bmod \, {{p}^{r}}).
\end{split}
\end{equation}

Combining  (\ref{thm2mid}) with (\ref{2ndsum}) and apply Theorem \ref{thm1}, we deduce
\begin{equation}\nonumber
\begin{split}
\sum\limits_{\begin{smallmatrix}
 i+j+k=m{{p}^{r}} \\
 i,j,k\in {\mathcal{P}_{p}}
\end{smallmatrix}}^{{}}{\frac{(-1)^i}{ijk}}
& = \Big(\big[\frac{m+1}{2}\big]+\big[\frac{m}{ 2}\big]\Big)\sum\limits_{\begin{smallmatrix}
 {{i}_{0}}+{{j}_{0}}+{{k}_{0}}={{p}^{r}} \\
 {{i}_{0}},{{j}_{0}},{{k}_{0}}\in {\mathcal{P}_{p}}
\end{smallmatrix}} {\frac{(-1)^{i_{0}}}{{{i}_{0}}{{j}_{0}}{{k}_{0}}}}\\
& \equiv \frac{m{{p}^{r-1}}}{2}{{B}_{p-3}}\,(\bmod \, {{p}^{r}}),
\end{split}
\end{equation}
hence (\ref{dengjia}) is true.
\end{proof}

\section{Concluding Remarks}
We leave two open questions for future research.

\textbf{Question 1.}  Can we find an arithmetical function $f(n)$ such that
\[\sum\limits_{\begin{smallmatrix}
 i+j+k=n \\
 i,j,k\in {\mathcal{P}_{n}}
\end{smallmatrix}}{\frac{{(-1)}^{i}}{ijk}}\equiv f(n)\,(\bmod \,n).\]

\textbf{Question 2.} For $n \ge 4$, can we find some analogous result for the following sum
\[\sum\limits_{\begin{smallmatrix}
 {{i}_{1}}+{{i}_{2}}+\cdots +{{i}_{n}}={{p}^{r}} \\
 {{i}_{1}},{{i}_{2}},\cdots ,{{i}_{n}}\in {\mathcal{P}_{p}}
\end{smallmatrix}}{\frac{{{(-1)}^{{{i}_{1}}}}}{{{i}_{1}}{{i}_{2}}\cdots {{i}_{n}}}}\]
modulo $p^{r}$?


\end{document}